\documentclass[12pt,twoside,leqno]{article}
\usepackage{amsmath}
\usepackage{amssymb}
\usepackage{amsxtra}
\usepackage{amscd}
\usepackage{amsthm}
\usepackage[mathscr]{eucal}
\usepackage{amsmath,amssymb,latexsym,amsfonts,graphicx}
\usepackage{array}
\usepackage[all]{xy}
\usepackage{supertabular}
\theoremstyle{plain}
\newtheorem{thm}[subsection]{Theorem}
\newtheorem{prop}[subsection]{Proposition}
\newtheorem{cor}[subsection]{Corollary}
\newtheorem{lem}[subsection]{Lemma}

\theoremstyle{definition}
\newtheorem{defn}[subsection]{Definition}

\input cyracc.def 

\newcommand{\Q}{\mathbb Q}

\newcommand{\C}{\mathbb C}
\newcommand{\cO}{{\mathcal O}_K}
\newcommand{\cE}{{\cal E}_1^*}
\newcommand{\cM}{{\mathcal M}_n}
\newcommand{\Z}{\mathbb Z}

\begin{document}

\title{On the 2-part of the Birch-Swinnerton-Dyer conjecture for elliptic curves with complex multiplication}

\author{John Coates,
Minhyong Kim \thanks{ Supported by EPSRC grant EP/G024979/2},
Zhibin Liang \thanks{Supported by NSFC11001183 and NSFC11171231.},
Chunlai Zhao\thanks{Supported by NSFC01272499.}}

\maketitle

 \newblock{\it To Peter Schneider for his 60th birthday}

\section{Introduction}

Let $E$ be an elliptic curve defined over $\Q$, with complex multiplication by the ring of integers of an imaginary quadratic field $K$. Thus, by the theory of complex multiplication, $K$ must be either $\Q(\sqrt{-1}), \Q(\sqrt{-2}), \Q(\sqrt{-3})$, or one of the fields
\begin{equation}\label{kay}
 \Q(\sqrt{-q}) \, \, (q=7, 11, 19, 43, 67, 163).
\end{equation}
Recently, Y. Tian \cite{T1}, \cite{T2} made the remarkable discovery that one could prove deep results about the arithmetic of certain quadratic twists of $E$ with root number $-1$, by combining formulae of Gross-Zagier type for these twists,  with a weak form of the 2-part of the conjecture of Birch and Swinnerton-Dyer for certain other  quadratic twists of $E$, where the root number is $+1$. We recall that, when the complex $L$-series of an elliptic curve with complex multiplication does not vanish at $s=1$, the $p$-part of the conjecture of Birch and Swinnerton-Dyer has been established,  by the methods of Iwasawa theory, for all primes $p$ which do not divide the order of the group of roots of unity of $K$ (see \cite{R}). However, at present we do not know how to extend such methods to cover the case of the prime $p=2$. Nevertheless, when $K = \Q(\sqrt{-1})$, one of us \cite{Z1}, \cite{Z2}, \cite{Z3}, \cite{Z4} did establish a weaker result in this direction for the prime $p=2$, by combining the classical expression for the value of the complex $L$-series as a sum of Eisenstein series (see Corollary \ref{z}), with an averaging argument over quadratic twists, and happily this weaker result has sufficed for Tian's work in \cite{T1}, \cite{T2}. The aim of the present note is to show that the rather elementary method developed in the
papers \cite{Z1}, \cite{Z2}, \cite{Z3}, \cite{Z4} works even more simply for quadratic twists of those elliptic curves $E$ having good reduction at the prime 2, and with complex multiplication by the ring of integers of the fields $K$ given by \eqref{kay}. We hope that one can use some of the weak forms of the 2-part of the conjecture of Birch and Swinnerton-Dyer established here (see, in particular, our Corollary \ref{ap}) to extend the deep results of \cite{T1}, \cite{T2}, \cite{T3},  to certain infinite families of quadratic twists of our curves $E$, having root number equal to $-1$ .  It is also interesting to note that, in \cite{T3}, Tian and his collaborators introduce a new and completely different method for establishing weak forms of the 2-part  part of the conjecture of Birch and Swinnerton-Dyer for curves with $K = \Q(\sqrt{-1})$, by using a celebrated formula of Waldspurger, and they believe that this new method can eventually be applied to a much wider class of elliptic curves, including those without complex multiplication. Needless to say, the rather elementary methods used here seem to be special to elliptic curves with complex multiplication. Finally, we wish to thank Y. Tian for his ever helpful comments on our work.

\section{The averaging argument}

Let  $K$ be an imaginary quadratic field of class number 1, which we assume is embedded in $\C$, and let $ \mathcal{O}_K$ its ring of integers. Let  $E$ be any elliptic curve defined over $K$, whose endomorphism ring is isomorphic to $\mathcal{O}_K$.  Fix once and for all a global minimal generalized Weierstrass equation for $E$ over $\mathcal{O}_K$
\begin{equation}\label{1}
 y^2+a_1xy+a_3y =
x^3+a_2x^2+a_4x+a_6 \, \, (a_i\in\mathcal{O}_K).
\end{equation}
Let $\frak L$ be the period lattice of the Neron differential
$dx/(2y+a_{1}x+a_{3})$.  Then $\frak L$ is a free
$\mathcal{O}_{K}$-module of rank 1, and we fix $\Omega_{\infty}\in\mathbb{C}^{\times}$ such that
$\frak {L}=\Omega_{\infty}\mathcal{O}_{K}$. Denote by $\psi_E$ the Grossencharacter of $E/K$ in the sense of Deuring-Weil, and write  $\frak f$ for the conductor of  $\psi_E$
(thus the prime divisors of $\frak f$ are precisely  the primes of $K$ where $E$ has bad reduction).
Now let $\frak g$ be any integral multiple of $\frak f$, and fix $g \in \cO$ such that $\frak g = g\cO$.
Let $S$ be the set of primes ideals of $K$ dividing $\frak g$, and denote by
$$
L_{S}(\bar{\psi}_{E},s)=\displaystyle\sum_{(\mathfrak{a}, \frak g)=1}\frac{\bar{\psi}_{E}(\mathfrak{a})}{(N\mathfrak{a})^{s}}
$$
the imprimitive Hecke $L$-function of the complex conjugate Grossencharacter of  ${\psi}_{E}$.
Our subsequent induction argument is based on the following expression for $L_{S}(\bar{\psi}_{E},s)$, which goes back to the 19th century. Let $z$ and $s$ be complex variables. For  any lattice $L$ in the complex plane $\mathbb{C}$, define the Kronecker-Eisenstein series by
$$
H_{1}(z,s,L):=\sum_{w\in L}\frac{\bar{z}+\bar{w}}{|z+w|^{2s}},
$$
where the sum is taken over all $w \in L$, except  $-z$ if $z \in $L. This series converges to define a holomorphic function of $s$ in the half plane $Re(s)>3/2$, and it has an analytic continuation to the whole $s$-plane. Let $\frak R$ denote the ray class field of $K$ modulo $\frak g$, and let $\cal B$
be any set of integral ideals of $K$, prime to $\frak g$, whose Artin symbols give precisely the Galois group of $\frak R$ over $K$ (in other words, $\cal B$ is a set of integral ideals of $K$ representing the ray class group of $K$ modulo $\frak g$). Since the conductor of  $\psi_E$ divides
$\frak g$, it is well known that $\frak R$ is equal to the field $K(E_g)$, which is obtained by adjoining to $K$ the coordinates of the $g$-division points on $E$.

\begin{prop}\label{2}  We have
\[ L_{S}(\bar{\psi}_E,s) =
\frac{|\Omega_\infty/g|^{2s}}{\overline{(\Omega_\infty/g)}}\displaystyle\sum_{\frak b\in\cal{B}}H_1(\psi_E(\frak b)\Omega_\infty/g,s,\frak{L}).\]
\end{prop}

\begin{proof}
As mentioned above $\cal B$ is a set of integral representatives of the ray class group of $K$ modulo $\frak g$, and so it follows that, fixing any generator of each $\frak b$ in $\cal B$, we obtain a set of representatives of
$(\mathcal{O}/\frak{g})^*/\tilde{\mu}_K,$  where $\tilde{\mu}_K$ denotes the image under reduction modulo $\frak{g}$ of the group $\mu_K$ of roots of unity of $K$. Moreover, the very existence of $\psi_E$ shows that the reduction map from $\mu_K$ to $\tilde{\mu}_K$ must be an isomorphism of groups. For each $\frak b$ in $\cal B$, we choose the generator of $\frak b$ given by $\psi_E(\frak b)$.  It follows that, as $\frak b$ runs over $\cal B$ and $c$ runs over $\frak g$,
the principal ideals $(\psi_E(\frak b)+c)$ run over all integral ideals of $K$, prime to $\frak g$, precisely once. Thus
$$
L_{S}(\bar{\psi}_E,s) =
\sum_{\frak b\in\cal{B}}\sum_{c\in \frak{g}}\frac{\bar{\psi}_E((\psi_E(\frak b)+c))}{|\psi_E(\frak b)+c|^{2s}}.
$$
Note that, since $c \in \frak{g}$, we have
$$
(\psi_E(\frak b)+c) = (\psi_E(\frak b))(1+ c/{\psi_E(\frak b)}) =
\frak{b}(1+ c/{\psi_E(\frak b)}),
$$
so that
$$
\psi_E((\psi_E(\frak b)+c)) =
\psi_E(\frak b)(1+ c/{\psi_E(\frak b)}) = \psi_E(\frak b)+c.
$$
Hence
 $$
L_{S}(\bar{\psi}_E,s) =
\sum_{\frak b \in B}\sum_{c \in \frak{g}}\frac{\overline{\psi_E(\frak b)+c}}{|\psi_E(\frak b)+c|^{2s}},
$$
which can easily be rewritten as
$$
\frac{|\Omega_{\infty}/g|^{2s}}{\overline{(\Omega_{\infty}/g})}\sum_{\frak {b} \in\cal B}\sum_{w \in \frak{L}}\frac{\overline{\psi_{E}\frak b)\Omega_{\infty}/g+w}}{|\psi_{E}(\frak b)\Omega_{\infty}/g+w|^{2s}} \, ,
$$
completing the proof of the theorem. \end{proof}

We recall that, for any lattice $L$, the non-holomorphic Eisenstein series ${\cal E}_1^*(z, L)$ is defined by
$$
\cE(z, L) = H_1(z, 1, L).
$$
Then the above proposition immediately implies that
$$
L_{S}(\bar{\psi}_E,1)/\Omega_\infty = g^{-1}\sum_{\frak {b} \in \cal B}\cE(\psi_E(\frak b)\Omega_\infty/g, \frak L).
$$
Also, it is well known (see, for example, \cite{GS}) that $\cE(\psi_E(\frak b)\Omega_\infty/g, \frak L)$ belongs to the field $\frak R$, and satisfies
$$
\cE(\psi_E(\frak b)\Omega_\infty/g, \frak L) = \cE(\Omega_\infty/g, \frak L)^{\sigma_{\frak b}},
$$
where $\sigma_{\frak b}$ denotes the Artin symbol of $\frak b$ in the Galois group of $\frak R$ over $K$. Thus the above proposition has the following immediate corollary, where  $Tr_{\frak R/K}$ denotes the trace map from $\frak R$ to $K$.

\begin{cor}\label{z} We have
$$
L_{S}(\bar{\psi}_E,1)/\Omega_\infty = Tr_{\frak R/K}(g^{-1}\cE(\Omega_\infty/g, \frak L)).
$$
\end{cor}

We next consider the twisting of $E$ by certain quadratic extensions of $K$. A non-zero element
$M$ of $\mathcal{O}_K$ is said to be square free if it is not divisible by the square of any non-unit element of this ring.

\begin{lem} Let $M$ be any non-zero and non-unit element of  $\mathcal{O}_K$ , which satisfies
(i) $M$ is square free, (ii) $M$ is prime to the discriminant of $K$, and (iii) $M \equiv 1 \, \, mod \, 4.$
Then the extension $K(\sqrt M)/K$ has conductor equal to $M\mathcal{O}_K$.
\end{lem}
\begin{proof} Since $M$ is square free and $M \equiv 1 \, \, mod \, 4$,
the extension  $K(\sqrt M)/K$ is totally and tamely ramified at all primes dividing $M$. Thus the assertion of the lemma will follow once we have shown that the primes of $K$ above 2 are not ramified in this extension. Let $v$ be any place of $K$ above 2. Let $w$ be such that $w^2 = M$, and put $z = (w-1)/2$. Then $z$ is a root of the polynomial $f(X) = X^2 - X - (M-1)/4$, so that $z$ is an algebraic integer. But $f'(z) = 2z-1$ is then clearly a unit at $v$, and so $v$ is unramified in our extension $K(\sqrt M)/K$, completing the proof.
\end{proof}

Let $M$ be as in the above lemma, and assume in addition that $(M, \frak f)=1$. We write $\chi_M$ for the abelian character of $K$ defining the quadratic extension $K(\sqrt M)/K$, and let $E^{(M)}$ denote the twist of $E$ by $\chi_M$. Thus
$E^{(M)}$ is the unique elliptic curve defined over $K$, which is isomorphic to $E$ over $K(\sqrt M)$, and which is such that
$$
E^{(M)}(K) = \{P \in E(K(\sqrt M))\, :   \sigma(P) = \, \chi_M(\sigma)(P), \, \sigma \in Gal(K(\sqrt M)/K)    \}.
$$
The curve $E^{(M)}$ also has endomorphism ring isomorphic to $\cO$, and its Grossencharacter,
which we denote by $\psi_{E^{(M)}}$, is equal to the product $\psi_E \chi_M$.  We write $\frak f_M$
for the conductor of  $\psi_{E^{(M)}}$. In view of the above lemma, we have  $\frak f_M = M\frak f$, because $(\frak f, M)=1$ and $\chi_M$ has conductor $M\cO$.  Finally, putting
$$
\frak p(z, \frak L) = x + (a_1^2 + 4a_2)/12, \, {\frak p}'(z, \frak L) = 2y + a_1x + a_3,
$$
we obtain a classical Weierstrass equation for $E$ over $\C$ of the form
$$
Y^2 = 4X^3 - g_2(\frak L)X - g_3(\frak L),
$$
with $X = \frak p(z, \frak L), \, Y = {\frak p}'(z, \frak L)$. The corresponding classical Weierstrass equation for $E^{(M)}$ over $\C$ is then given by
$$
Y^2 = 4X^3 - M^2g_2(\frak L)X - M^3g_3(\frak L).
$$
Hence the period lattice for the curve  $E^{(M)}$ over $\C$ is given by
\begin{equation} \label{6}
{\frak L}_M = \frac{ \Omega_{\infty}}{\sqrt M}\cO.
\end{equation}

\medskip

We now suppose that we are given an infinite sequence
$$
\pi_1, \pi_2, \ldots, \pi_n, \ldots
$$
of distinct prime elements of $\cO$. We shall say that this sequence is {\it admissible} for $E/K$  if, for all $n \geq 1$, we have that $\pi_n$ is  prime to the discriminant of  $K$, and
\begin{equation} \label{8}
\pi_n \equiv 1 \, \, mod \, 4,  \, \, \, \,   (\pi_n, \frak f) = 1.
\end{equation}
For each integer $n \geq 0$, define
\begin{equation}\label{3}
\cM = \pi_1\cdots \pi_n, \, \, \, \frak g_n = \cM \frak f.
\end{equation}
We now take $\frak R_n$ to be the ray class field of $K$ modulo $\frak g_n$.  Since  $\pi_j \equiv 1 \, \, mod \, 4$, the above lemma shows that the extension $K(\sqrt{\pi_j})/K$ has conductor $\pi_j \cO$, and so is contained in $\frak R_n$, for all $j$ with
$1\leq j \leq n$. Hence the field $\frak J_n$ defined by
\begin{equation}\label{4}
\frak J_n = K(\sqrt{\pi_1},..., \sqrt{\pi_n})
\end{equation}
is always a subfield of $\frak R_n$. Let $S_n$ be the set of prime ideals of $K$ dividing $\frak g_n$. Also, writing $f$ for any $\cO$ generator of the ideal $\frak f$, we put $g_n = f\cM$, so that $\frak g_n = g_n\cO$. Finally, we define $\mathcal D_{n}$ to be the set of all divisors of $\cM$ which are given by the product of any subset of $\{\pi_1,..., \pi_n\}$. The averaging theorem which follows is essentially contained in the earlier paper of one of us \cite{Z1}, and is the basis of all of our subsequent arguments. For simplicity, we write just $\psi_M$ for the Grossencharacter of the curve $E^{(M)}$ for any $M \in \mathcal D_n$.

\begin{thm}\label{av} Let $\{\pi_1,..., \pi_n,...\}$ be any admissible sequence for $E/K$. Then, for all integers $n \geq 1$, we have
\begin{equation} \label{5}
\sum_{M \in {\mathcal D}_{n}} L_{S_n}(\bar{\psi}_M, 1)/{\Omega_{\infty}}  = 2^nTr_{\frak R_n/\frak J_n}({g_n}^{-1}\cE(\Omega_\infty/g_n, \frak L)),
\end{equation}
where $Tr_{\frak R_n/\frak J_n}$ denotes the trace map from $\frak R_n$ to $\frak J_n$.
\end{thm}

\begin{proof} Let $M$ be any element of  $\mathcal D_{n}$. Applying Corollary \ref{z}
to the curve $E^{(M)}$ with $\frak g = \frak g_n$, and using \eqref{6}, we conclude that
$$
 L_{S_n}(\bar{\psi}_M, 1)\sqrt{M}/{\Omega_{\infty}} = Tr_{\frak R_n/K}({g_n}^{-1}\cE(\frac{\Omega_\infty}{\sqrt{M}g_n}, \frak L_M)).
 $$
Now, for any non-zero complex number $\lambda$, we have
$$
\cE(z, \frak L_M) = \lambda \cE(\lambda z, \lambda \frak L_M).
$$
Hence, taking $\lambda = \sqrt{M}$, and writing $G_n$ for the Galois group of $\frak R_n/K$, we conclude that
\begin{equation}\label{7}
 L_{S_n}(\bar{\psi}_M, 1)/{\Omega_{\infty}} = \sum_{\sigma \in G_n}(\sqrt{M})^{\sigma - 1}{g_n}^{-1}(\cE(\Omega_\infty/g_n, \frak L))^\sigma.
 \end{equation}
It is now clear that the assertion of the theorem is an immediate consequence of the following lemma.
\end{proof}

\begin{lem}\label{div} Let $H_n = Gal({\frak R_n}/{\frak J}_n)$. If $\sigma$ is any element of $G_n$, then
$\sum_{M \in {\mathcal D}_n}(\sqrt{M})^{\sigma - 1}$ is equal to $2^n$ if $\sigma$ belongs to $H_n$, and is equal to $0$ otherwise.
\end{lem}
\begin{proof} The first assertion of the lemma is clear.  To prove the second assertion, suppose that
$\sigma$ maps $k \geq 1$ elements of the set  $\{\sqrt{\pi_1},..., \sqrt{\pi_n\}}$ to minus themselves, and write $V(\sigma)$ for the subset consisting of all such elements.  If $M$ be any element of $\mathcal D_n$, it is clear that $\sigma$ will fix $\sqrt M$ if and only if $M$ is a product of an even number of elements of $V(\sigma)$, with an arbitrary number of elements of the complement of $V(\sigma)$ in $\{\sqrt{\pi_1},..., \sqrt{\pi_n}\}$. Hence the total number of $M$ in $\mathcal D_n$ such that $\sigma$ fixes $\sqrt M$ is equal to
$$
2^{n-k}( (k, 0) + (k, 2) + (k, 4) + \ldots) = 2^{n-1},
$$
where $(n,r)$ denotes the number of ways of choosing $r$ objects from a set of $n$ objects. Similarly, the total number of $M$ in $\mathcal D_n$ such that $\sigma$ maps  $\sqrt M$ to $- \sqrt M$ is equal to
$$
2^{n-k}((k,1) + (k, 3) + (k, 5) + \dots) = 2^{n-1}.
$$
Since these last two expressions are equal, the second assertion of the lemma is now clear.
\end{proof}

\section{Integrality at 2}

We use the notation and hypotheses introduced in the last section. Our aim in this section is to prove the following result.

\begin{thm}\label{int} Assume that $E$ has good reduction at the primes of K above 2, and that $\{\pi_1,..., \pi_n,...\}$ is any admissible sequence for $E/K$.  For all $n \geq 1$,  define
$$
\Psi_n =Tr_{\frak R_n/\frak J_n}({g_n}^{-1}\cE(\frac{\Omega_\infty}{g_n}, \frak L)).
$$
Then $2\Psi_n$ is always integral at all places of $\frak J_n$ above 2. Moreover, if the coefficient $a_1$ in \eqref{1} is divisible by 2 in $\cO$, then $\Psi_n$ is integral at all places of $\frak J_n$ above 2.

\end{thm}

\medskip

Before giving the proof of the theorem, we recall some classical identities involving elliptic functions (see. for example, \cite{D}). Let $L$ be any lattice in the complex plane, and write $\frak p(z, L)$ for the Weierstrass $\frak p$ -function attached to $L$. For each integer $m \geq 2$, we define the elliptic function $B_m(z,L)$ by
$$
2B_m(z, L) = \frac{\frak p''(z,L)}{\frak p'(z, L)} + \sum _{k=2}^{k=m-1} \frac{\frak p'(kz, L) -\frak p'(z,L)}{\frak p(kz, L) -\frak p(z,L)}.
$$
\begin{lem} For all integers $m \geq 2$, we have
$$
B_m(z, L) = \cE(mz, L) - m\cE(z, L).
$$
\end{lem}
\begin{proof} Let $\zeta(z, L)$ denote the Weierstrass zeta function of $L$.  The following identity is classical
$$
\cE(z, L) = \zeta(z, L) - z s_2(L) - \bar{z}A(L)^{-1},
$$
(see, for example, Prop. 1.5 of \cite{GS}, where the definitions of the constants $s_2(L)$ and $A(L)$ are also given). It follows immediately that
$$
\cE(mz, L) - m\cE(z, L) = \zeta(mz, L) - m\zeta(z, L).
$$
But now we have the addition formula
$$
\zeta(z_1+z_2, L) = \zeta(z_1, L) + \zeta(z_2, L) + \frac{1}{2} \frac{\frak p'(z_1, L) -\frak p'(z_2,L)}{\frak p(z_1, L) -\frak p(z_2,L)}.
$$
Taking the limit as $z_1$ tends to $z_2$, we obtain the statement of the lemma for $m=2$. For any $m \geq 2$, the above addition formula also shows that
$$
 \zeta((m+1)z, L) - (m+1)\zeta(z, L) =  \zeta(mz, L) - m\zeta(z, L) +  \frac{1}{2} \frac{\frak p'(mz, L) -\frak p'(z,L)}{\frak p(mz, L) -\frak p(z,L)},
 $$
whence the assertion of the lemma follows by induction on $m$.
\end{proof}

The next lemma is attributed in \cite{D} to unpublished notes of Swinnerton-Dyer.

\begin{lem} Let $w$ be any complex number such that $w$ is not in $L$, but $mw$ does belong to $L$ for some integer $m\geq 2$. Then $\cE(w, L) = -B_{m-1}(w, L)/m$.
\end{lem}
\begin{proof} By the previous lemma, we have
$$
B_{m-1}(w, L) = \cE((m-1)w, L) - (m-1)\cE(w, L).
$$
But, as a function of $z$,  $\cE(z,L)$ is periodic with respect to $L$ and odd, whence it follows that
$\cE((m-1)w, L) = -\cE(w, L)$. This completes the proof.
\end{proof}

Now we have the addition formula
$$
\frak p(z_1+z_2, L)  + \frak p(z_1, L) + \frak p(z_2, L) =  \frac{1}{4} ((\frak p'(z_1, L) -\frak p'(z_2,L))/(\frak p(z_1, L) -\frak p(z_2,L)))^2,
$$
whence we immediately obtain the following corollary.

\begin{cor}\label{11}  Let $w$ be any complex number such that $w$ is not in $L$, but $w$ does have finite order in $\C/L$.   Let  $m$ be the exact order of $w$ in $\C/L$. Assuming $m \geq 3$, we have
$$
m\cE(w,L) = \sum _{k=1}^{k=m-2}\epsilon_k(\frak p((k+1)w, L) + \frak p(kw, L) + \frak p(w, L))^{1/2},
$$
where $\epsilon_k$ denotes the sign $+1$ or $-1$.
\end{cor}

\medskip

We can now give the proof of Theorem \ref{int}. Recall that the period lattice of the Neron differential of our fixed global minimal Weierstrass equation \eqref{1} is $\frak L= \Omega_{\infty}\cO$. Take $w= \psi (\frak b)\Omega_{\infty}/g_n$, where $\frak b$ is any fixed integral ideal of $K$ prime to $\frak g_n$. Thus $\cE(w, \frak L)$ is any one of the conjugates of $\cE(\Omega _\infty/g_n, \frak L)$ over $K$.  Let $m$ be the smallest positive rational integer lying in the ideal $\frak g_n$, so that $m$ is also the smallest positive rational integer with the property that $mw$ lies in $\frak L$. Moreover, since $E$ has good reduction at the primes of $K$ above 2, the ideal $\frak f$ is not divisible by any prime of $K$ above 2. This means that the smallest positive rational integer in the ideal $\frak g_n$ must be odd.  It follows that $m$ is odd, and it must then be $>2$. Let $P$ be the point on $E$ defined by $w$. Then we have
\begin{equation}\label{12}
{\frak p}(rw, \frak L) = x(rP) + (a_1^2 + 4a_2)/12, \, \, (r=1,..., m-1).
\end{equation}
But, as $E$ has good reduction at all primes of $K$ above 2 and the point $rP$ has odd order,  it follows that $x(rP)$ is integral at each prime
of $\frak R_n$ above 2. Thus we can immediately conclude from Corollary \ref{11} and \eqref{12} that the following two assertions. Firstly, if $a_1/2$ lies in $\cO$, then
every conjugate of $\cE(\Omega_{\infty}/g_n, \frak L)$ over $K$ is integral at all places of $\frak R_n$ above 2. In general, if we drop the assumption that $a_1/2$ lies in $\cO$, all we can say is that every conjugate of $2\cE(\Omega_{\infty}/g_n, \frak L)$ over $K$ is integral at every place of $\frak R_n$ above $2$. Taken together, these two assertions clearly imply Theorem \ref{int}. \qed.

\section{The induction argument}

Let $E$ be an elliptic curve defined over $K$, with complex multiplication by the ring of integers of $K$, and global minimal Weierstrass equation given by \eqref{1}. We fix once and for all any place of the algebraic closure of $\Q$ above 2, and write $ord_2$ for the order valuation at this place, normalized so that $ord_2(2) = 1$. Define $\alpha_E$ to be 0 or 1, according as 2 does or does not divide $a_1$ in $\cO$, where we recall that $a_1$ is one of the coefficients in the equation \eqref{1}.  For any admissible sequence $\{\pi_1,..., \pi_n,...\}$ for $E/K$,  we define $\frak M_n = \pi_1\ldots\pi_n$, and
\begin{equation}\label{13}
L^{(alg)}(\bar{\psi}_{\frak M_n}, 1) = L(\bar{\psi}_{\frak M_n}, 1)\sqrt{\frak M_n}/\Omega_\infty,
\end{equation}
which is an element of $K$. Moreover, we define
\begin{equation}\label{13*}
  \phi_E = \alpha_E \, \, \rm{or} \, \, \, max\{\alpha_E, \, -ord_2(\it{L^{(alg)}}(\bar{\psi}_E, 1))\},
\end{equation}
according as $L(\bar{\psi}_E, 1) = 0$, or $L(\bar{\psi}_E, 1) \neq 0$.  Our goal in this section is to prove the following theorem.

\begin{thm}\label{21} Assume that $K \neq \Q(\sqrt{-1}), \Q(\sqrt{-3})$, and that $E$ has good reduction at all places of $K$ above 2. Then, for all admissible sequences $\{\pi_1,..., \pi_n,...\}$ for $E/K$, and all integers $n \geq 1$, we have
\begin{equation}\label{14}
ord_2(L^{(alg)}(\bar{\psi}_{\frak M_n}, 1)) \geq n - \phi_E.
\end{equation}
\end{thm}

\begin{proof} We shall prove the theorem by induction on $n$, and we begin with an obvious remark.  Let $r$ be any integer $\geq 0$, and recall that $\psi_{\frak M_r}$ denotes the Grossencharacter of the twisted curve $E^{(\frak M_r)}$. For each $n > r$, write $\frak p_n = \pi_n\cO$. Then $\frak p_n$ is prime to the conductor of $\psi_{\frak M_r}$, and we have
\begin{equation}\label{15}
ord_2(1- \bar{\psi}_{\frak M_r}( \frak p_n)/N \frak p_n)  \geq 1.
\end{equation}
Indeed, we have $\psi_{\frak M_r}( \frak p_n) = \zeta \pi_n$, where $\zeta = 1$ or $-1$
because $K \neq \Q(\sqrt{-1}), \Q(\sqrt{-3})$. Thus $\zeta \equiv 1\, mod \, \, 2$, and \eqref{15}
then follows easily because $\pi_n \equiv 1 \, mod \, 4$ and $ N \frak p_n = \psi_{\frak M_r}( \frak p_n)  \bar{\psi}_{\frak M_r}( \frak p_n).$ Note also that, on combining Theorems \ref{int} and \ref{av},
we conclude that, for all integers $n \geq 1$, we have
\begin{equation}\label{16}
ord_2(\sum_{M \in {\mathcal D}_{n}} L_{S_n}(\bar{\psi}_M, 1)/\Omega_{\infty}) \geq n - \alpha_E.
\end{equation}
It is clear that, on combining \eqref{15} for $r=0$ and \eqref{16} for $n=1$, we immediately obtain
\eqref{14} for $n=1$. Suppose now that $n > 1$, and that \eqref{14} has been proven for all integers
strictly less than $n$. Combining this inductive hypothesis with assertion \eqref{15}, we conclude that for all proper divisors $M$ of $\frak M_n$, we have
$$
 ord_2(L_{S_n}(\bar{\psi}_M, 1)/\Omega_\infty) \geq n - \phi_E,
 $$
 whence \eqref{16} again shows that \eqref{14} holds for the integer $n$. This completes the proof of the theorem.
\end{proof}

We next investigate which rational primes $p$ split in $K$, and have the additional property that they can be written as $p=\pi\pi^*$, with $\pi$ in $\cO$ satisfying $\pi \equiv 1 \, mod  \, 4$ (and thus automatically also satisfying $\pi^* \equiv 1 \, mod  \, 4$). We call primes $p$ with this property {\it special} split primes for $K$. Obviously, a necessary condition for $p$ to be a special split prime for $K$ is that $p \equiv 1 \, mod \, 4$. We remark that it is clear from the Chebotarev density theorem that there are always infinitely many special split primes for $K$.

\begin{lem}\label{split} Assume that  $K \neq \Q(\sqrt{-1}), \Q(\sqrt{-2}), \Q(\sqrt{-3})$. Let $p$ be any rational prime which splits in $K$, and which satisfies $p \equiv 1\, mod \, 4$. If $K = \Q(\sqrt{-7})$, then $p$ is always a special split prime for $K$. If $K = \Q(\sqrt{-q})$, where $q = 11, 19, 43, 67, 163$, then such a $p$ is a special split prime for $K$ if and only if we can write $p = \pi\pi^*$ in $\cO$ with $\pi + \pi^* \equiv 0 \, mod \, 2$.
\end{lem}
\begin{proof} Let  $K = \Q(\sqrt{-q})$,  and put $\tau = (1+\sqrt{-q})/2$, so that $1, \tau$ form an integral basis of $\cO$. Assume first that $K = \Q(\sqrt{-7})$. Then $p = a^2 + ab + 2b^2$, with $a$ an odd integer, whose sign can be chosen so that $a\equiv 1 \, mod \, 4$, and with $b$ an even integer, which has necessarily to be divisible by 4 since $p \equiv 1 \, mod \, 4$. We then clearly have that
$ \pi = a + b\tau$ satisfies  $\pi \equiv 1\,  mod \, 4$. Finally, assume that $K = \Q(\sqrt{-q})$, where $q$ is any of $11, 19, 43, 67, 163$. Then $p = a^2 + ab + mb^2$, where $a$ and $b$ are integers, and $m  = (q+1)/4$ is now an odd integer. Since $p \equiv 1\, mod \, 4$, we see that $\pi = a + b\tau$ satisfies $\pi \equiv 1 \, mod \, 4$ if and only if $a \equiv 1\, mod \, 4 \,$ and $b$ is even. But $\pi + \pi^* = 2a + b$, and so $\pi + \pi^*$ will be even if and only if $b$ is even. By if $b$ is even, then $a$ is odd, and then we can always choose the sign of $a$ so that  $a \equiv 1\, mod \, 4$. This completes the proof.
\end{proof}

\medskip

Now assume that our elliptic curve $E$ is in fact defined over $\Q$, and take \eqref{1} to be a global minimal Weierstrass equation for $E$ over $\Q$. Then the conductor $N(E)$ of $E$ is given by 
$$
N(E) =  d_KN\frak f, 
$$
where $d_K$ denotes the absolute value of the discriminant of $K$. Moreover, the complex $L$-series $L(E, s)$ of $E$ over $\Q$  coincides with the Hecke L-seres  $L(\bar{\psi}_E, s)$. If $R$ is a non-zero square free integer, $E^{(R)}$ will now denote the twist of $E$ by the extension $\Q(\sqrt{R})/\Q$. Write
\begin{equation}\label{17}
L^{(alg)}(E^{(R)}, 1) = L(E^{(R)}, 1)\sqrt{R}/\Omega_\infty.
\end{equation}
Finally, $\alpha_E$ has the same definition as earlier, and $\phi_E$ is again defined by \eqref{13*}.

\medskip

\begin{lem}\label{sq1} Assume that $E$ is defined over $\Q$, and has complex multiplication by the ring of integers of any of the fields $K = \Q(\sqrt{-q})$, where $q = 7, 11, 19, 43, 67, 163$. Suppose further that $E$ has good reduction at 2. Then the conductor $N(E)$ of $E$ is a square.
\end{lem}
\begin{proof} Let $p$ be any prime dividing $N(E)$. Since $E$ has potential good reduction at $p$,
we must have that $p^2$ exactly divides $N(E)$ whenever $p > 3$. Also $p \neq 2$, because $E$ has good reduction at 2. Thus we only have to check that an even power of 3 must divide $N(E)$. But, since $q > 3$,  it is well known (see \cite{GR}) that $E$ is the quadratic twist of an elliptic curve of conductor $q^2$, whence it follows immediately that either 3 does not divide $N(E)$, or $3^2$ exactly divides $N(E)$, according as 3 does not, or does,  divide the discriminant of the twisting quadratic extension. This completes the proof. 
\end{proof}

We now introduce a definition which for the moment is motivated by what is needed to deduce the next theorem from our earlier induction argument (but see also the connexion with Tamagawa factors discussed in the next section). Write $w_E$ for the sign in the functional equation of $L(E,s)$. We continue to assume that $E$ is defined over $\Q$, and satisfies the hypotheses of Lemma \ref{sq1}.
If $D$ is any square free integer which is prime to $N(E)$, it is well known that the root
number of the twist $E^{(D)}$ of $E$ by the quadratic extension $\Q(\sqrt{D})/\Q$ is given by $\chi_D(-N(E))w_E$, where $\chi_D$ denotes the Dirichlet character of this quadratic extension. Thus, in view of Lemma \ref{sq1}, we are led to make the following definition. 

\begin{defn} Assume that $E$ satisfies the hypotheses of Lemma \label{sq}. A square free  positive integer $M$ is said to be {\it admissible} for $E$ if it satisfies (i) $(M, N(E))=1$, (ii) $M \equiv 1 \, mod \, 4$ or $M \equiv 3 \, mod \, 4$, according as $w_E = +1$ or $w_E = -1$, and (iii) every prime factor of $M$ which splits in $K$ is a special split prime for $K$.
\end{defn}

\begin{thm}\label{18} Assume that $E$ is defined over $\Q$, has complex multiplication by the ring of integers of $K=\Q(\sqrt{-q})$, where $q = 7, 11, 19, 43, 67, 163$, and has good reduction at 2.  Let $M$ be a square free positive integer, which is admissible for $E$, and let $r(M)$ denote the number of primes of $K$ dividing $M$. Put $\epsilon$ equal to $+1$ or $-1$, according as $M \equiv 1 \, or  \,  3 \, 
mod \, 4$. Then, for $w_E = \epsilon$, we have
\begin{equation}\label{19}
ord_2(L^{(alg)}(E^{(\epsilon M)}, 1)) \geq r(M) - \phi_E.
\end{equation}
\end{thm}
\begin{proof} Let $M$ be any square free integer which is admissible for $E$, and let $p$ be any prime dividing $M$. If $p$ is inert in $K$, define $\pi$ to be $p$ or $-p$, according as $p$ is congruent to 1 or 3 $mod \, 4$. If $p$ splits in $K$, then Lemma \ref{split} shows that we can then write $p = \pi\pi^*$, where $\pi$ and $\pi^*$ are elements of $\cO$, which are both congruent to 1 $mod \, 4$. Since every $p$ with $p \equiv 3 \, mod \, 4$, and $p$ dividing $M$, is inert in $K$, it is now clear that  we can write
$$
\epsilon M = \pi_1\ldots \pi_{r(M)}, \, \ ,
$$
where the $\pi_i$ are distinct prime elements of $\cO$, which are all congruent to 1 $mod \, 4$,
and which are also prime to $\frak f$ and the discriminant of $K$. Hence the above theorem is an immediate consequence of Theorem \ref{21}.
\end{proof}

The following is an immediate corollary of the above theorem. Of course, the hypothesis made in the corollary that $L(E, 1) \neq 0$ implies that the root number $w_E = 1$, and so the admissible $M$ in this case are $\equiv 1 \, mod \, 4.$

\begin{cor}\label{ap} Assume that $E$ is defined over $\Q$, has complex multiplication by the ring of integers of $K$, and has good reduction at 2.  Suppose further that  we have (i) $K \neq \Q(\sqrt{-3})$, (ii) $L(E, 1) \neq 0$, and (iii) $ord_2(L^{(alg)}(E, 1)) < 0$. Let $M$ be any square free positive integer which is admissible for $E$, and which is divisible only by rational primes which split in $K$. Then
$$
ord_2(\frac{L^{(alg)}(E^{(M)}, 1)}{L^{(alg)}(E, 1)}) \geq 2k(M),
$$
where $k(M)$ denotes the number of rational prime divisors of $M$.
\end{cor}

\medskip

We now discuss some numerical examples of this theorem. For basic information about the curves discussed below, see, for example, \cite{GR}. As a first example, let $E$ be the elliptic curve
defined by
\begin{equation}\label{22}
y^2 + xy = x^3 - x^2 - 2x - 1.
\end{equation}
It has conductor 49, and complex multiplication by the ring of integers of $K = \Q(\sqrt{-7})$.
In fact, this curve is isomorphic to the modular curve $X_0(49)$.  By the Chowla-Selberg formula, the period lattice $\frak L$ of the Neron differential on $E$ is given by $\frak L = \Omega_\infty\cO$, where
$$
\Omega_\infty = \frac{\Gamma(\frac{1}{7})\Gamma(\frac{2}{7})\Gamma(\frac{4}{7})}{2\pi i\sqrt{-7}}.
$$
Moreover, $\alpha_E = 1$ because $a_1 = 1$, and $L^{(alg)}(E, 1) = 1/2$, so that $\phi_E = 1$.
Note that any positive square free integer $M$ with $(M, 7) = 1$ and $M \equiv 1 \, mod \, 4$,
will be admissible for $E$, provided each of its prime factors which splits in $K$ (thus a prime factor which is congruent to any of  1, 2, or 4 $mod \, 7$) is congruent to 1 $mod \, 4$.  Theorem \ref{18}  therefore implies that, for such admissible integers $M$, we have
\begin{equation}\label{23}
ord_2(L^{(alg)}(E^{(M)}, 1)) \geq r(M) - 1.
\end{equation}
We see from Table I at the end of this paper that this estimate is in general best possible. As a second example,  take for $E$ the elliptic curve defined by
\begin{equation}\label{22'}
y^2 + y = x^3 - x^2 - 7x +10.
\end{equation}
It has conductor 121, and complex multiplication by the ring of integers of $K = \Q(\sqrt{-11})$.
Again by the Chowla-Selberg formula, the period lattice $\frak L$ of the Neron differential on $E$ is given by $\frak L = \Omega_\infty\cO$, where
$$
\Omega_\infty = \frac{\Gamma(\frac{1}{11})\Gamma(\frac{3}{11})\Gamma(\frac{4}{11})\Gamma(\frac{5}{11})\Gamma(\frac{9}{11})}{2\pi i\sqrt{-11}}.
$$
Moreover, $\alpha_E = 0$ because $a_1 = 0$, and $w_E = -1$, so that $\phi_E = 0$. The split primes for $K$ are those which are congruent to $1, 3, 4, 5, 9 \, mod \, 11$. For example, all special split primes $< 1000$ for this curve are:-
\begin{align*}53,257,269,397,401,421,617,757,773,929.\end{align*}
Let now $M$ be any square free positive integer which is admissible for $E$ (in particular, since we are only interested in twists $E^{(-M)}$ having root number equal to $+1$, we assume that $M \equiv 3 \, mod \, 4$ and $(M, 11) = 1$).  Then
Theorem \ref{18} implies that
\begin{equation}\label{24}
ord_2(L^{(alg)}(E^{(-M)}, 1)) \geq r(M).
\end{equation}
However, in this example, Table II at the end of this paper suggests that this estimate is not, in general, best possible. It seems plausible to speculate from Table II that the lower bound of $\eqref{24}$ could be improved to $r(M)+1$.

\section{Tamagawa Factors}

Our goal in this last section is to relate the estimate given by Theorem \ref{18} to the Tamagawa factors which arise in the Birch-Swinnerton-Dyer conjecture for the twists of our given elliptic curve with complex multiplication. Suppose first that $E$ is any elliptic curve $E$ defined over $\Q$, and any prime $p$ of bad reduction for $E$, let $E(\Q_p)$ denote the group of points on $E$ with coordinates in the field of $p$-adic numbers $\Q_p$, and  $E_0(\Q_p)$ the subgroup of points with non-singular reduction modulo $p$. We define
$$
\frak C_p(E) = E(\Q_p)/E_0(\Q_p),
$$
and recall that the Tamagawa factor $c_p(E)$ is defined by
\begin{equation}\label{25}
c_p(E) = [E(\Q_p):E_0(\Q_p)].
\end{equation}
If $A$ is any abelian group, $A[m]$ will denote the kernel of multiplication by a positive integer $m$
on $A$. The following lemma is very well known, but we give it for completeness.

\begin{lem}\label{use} Let $E$ be any elliptic curve over $\Q$, and let $p$ be a prime number where $E$ has bad additive reduction. Then, for all positive integers $m$ with $(m, p) = 1$, we have
$$
\frak C_p(E)[m] = E(\Q_p)[m].
$$
\end{lem}
\begin{proof} Let $E_1(\Q_p)$ denote the group of points on the formal group of $E$ at $p$.
Since $E$ has additive reduction modulo $p$, the group of non-singular points on the reduction of $E$ modulo $p$ is isomorphic to the additive group of the field $\mathbb{F}_p$. As $E_1(\Q_p)$  is pro-$p$, and we have the exact sequence
$$
0 \to E_1(\Q_p) \to E_0(\Q_p) \to \mathbb{F}_p \to 0,
$$
it follows immediately that multiplication by $m$ is an isomorphism on $E^0(\Q_p)$, whence the
assertion of the lemma follows easily from a simple application of the snake lemma to the sequence
$$
0 \to E_0(\Q_p) \to E(\Q_p) \to \frak C_p(E) \to 0.
$$
\end{proof}

\medskip

As earlier, let $E$ now be our elliptic curve  defined over $\Q$ with complex multiplication by the ring of integers of the imaginary quadratic field $K$, and write $N(E)$ for the conductor of $E$. Once again, we will assume that $E$ has good reduction at 2, and so we cannot have $K = \Q(\sqrt{-1})$,
or $K = \Q(\sqrt{-2})$.  Let  $M$ denote an odd positive square free integer with $(M, N(E)) =1$.
We put $\epsilon$ equal to $+1$ or $-1$, according as $M$ is congruent to 1 or 3 $mod \, 4$. Thus 2 is always unramified in the quadratic extension $\Q(\sqrt{\epsilon M})/\Q$.

\begin{lem}\label{eq} Let $p$ be any prime number dividing $N(E)$ or $M$. If $p$ divides $N(E)$, then $ord_2(c_p(E^{(\epsilon M)})) = ord_2(c_p(E))$. If $p$ divides $M$, then the value of $ord_2(c_p(E^{(\epsilon M)}))$ is independent of $M$.
\end{lem}
\begin{proof} Let $p$ be any prime factor of $N(E)$ or $M$, so that, in particular, $p$ is odd.  Since the $j$-invariant of $E$, and so also the $j$-invariant of $E^{(\epsilon M)}$, are integral, it follows from the table on p. 365 of \cite{S} that
the 2-primary subgroups of $\frak C_p(E)$ and $\frak C_p(E^{(\epsilon M)})$ are either $0, \Z/2\Z$, or $\Z/2\Z \times \Z/2Z$. Now when $p$ divides $N(E)$, both $E$ and $E^{(\epsilon M)}$ have additive reduction at $p$, and so we conclude from Lemma \ref{use} that, in this case,
$$
ord_2(c_p(E)) = ord_2(\#(E(\Q_p)[2])),  \, ord_2(c_p(E^{(\epsilon M)})) = ord_2(\#(E^{(\epsilon M)}(\Q_p)[2])),
$$
Also when $p$ divides $M$, we have, again from Lemma \ref{use}, that
\begin{equation}\label{cr}
ord_2(c_p(E^{(\epsilon M)})) = ord_2(\#(E^{(\epsilon M)}(\Q_p)[2])).
\end{equation}
But the Galois group of $\Q(\sqrt{\epsilon M})/\Q$ clearly acts trivially on points of order 2 on $E^{(\epsilon M)}$, and so we always have
\begin{equation}\label{cr'}
\#(E(\Q_p)[2]) = \#(E^{(\epsilon M)}(\Q_p)[2])).
\end{equation}
The assertions of the lemma now follow immediately.
\end{proof}

\medskip

\begin{thm}\label{26} Assume that $E$ is defined over $\Q$ and has good reduction at 2, and that $K \neq \Q(\sqrt{-3})$. Let $M$ be an odd positive square free integer with $(M, N(E))=1$, and having the property that every prime factor of M which is inert in $K$ is congruent to 1 $mod \, 4$. Let $p$ be any prime dividing M. Then (i) $ord_2(c_p(E^{(\epsilon M)})) = 1$
if $p$ is inert in $K$,  (ii) $ord_2(c_p(E^{(\epsilon M)})) = 0$ if $p$ splits in $K$ and the trace of the Frobenius endomorphism of the reduction of $E$ modulo $p$ is odd, and (iii) $ord_2(c_p(E^{(\epsilon M)})) = 2$ if $p$ splits in $K$ and the trace of the Frobenius endomorphism of the reduction of $E$ modulo $p$ is even.
\end{thm}
\noindent Before giving the proof of this theorem, we state an important corollary.

\begin{cor} Assume that $E$ is defined over $\Q$ and has good reduction at 2, and that $K \neq \Q(\sqrt{-3})$. Let $M$ be a positive integer which is admissible for $E$ in the sense of Definition \ref{ad}, and has the property that every prime factor of $M$ is congruent to 1 $mod \, 4$.  Write $r(M)$ for the number of primes divisors of $M$ in $K$. Then
\begin{equation}\label{27}
ord_2(\prod_{p|M} {c_p(E^{(\epsilon M)}})) = r(M).
\end{equation}
\end{cor}

\begin{proof}  Let $p$ be any prime factor of $M$. Since $(p, N(E))=1$, $E$ has good reduction at $p$ and $p$ does not ramify in $K$. Recalling \eqref{cr} and \eqref{cr'}, we have to compute the order of $E(\Q_p)[2]$. Now, since $p$ is odd,  the Galois module $E[2]$ is unramified at $p$. Let $\tilde{E}$ denote the reduction of $E$
modulo $p$. Since the formal group of $E$ at $p$ is a $\Z_p$-module, it follows easily that reduction modulo $p$ defines an isomorphism
$$
E(\Q_p)[2] = \tilde{E}(\mathbb{F}_p)[2].
$$
Now the order of  $\tilde{E}(\mathbb{F}_p)$ is $1+p$ or $1 - a_p + p$, according as $p$ is inert or splits in $K$, where $a_p$ is the trace of the Frobenius endomorphism of $\tilde{E}$. In particular, when $p$ splits in $K$ and $a_p$ is odd, we see immediately that $E(\Q_p)[2] =0$. Similarly, if $p$ is inert in $K$, then, as $p \equiv 1 \, mod \, 4$, we conclude that  $E(\Q_p)[2]$ must have order 2. Suppose next that $p$ splits in $K$ and $a_p$ is even. Let $\tau_p$ be any Frobenius automorphism at $p$.
Since $p$ splits in $K$, we can view $\tau_p$ as an element of the absolute Galois group of $K$,
and we write $\phi_p$ for its image in the $\cO$-automorphism group of the module $E[2]$, which is equal to
$(\cO/2\cO)^*$. Then $\phi_p$ must have order dividing 2 because, since $a_p$ is even, its characteristic polynomial is equal to  $X^2 - 1$. But 2 is not ramified in $K$ because $E$ has good reduction at 2. Thus
the group $(\cO/2\cO)^*$ has no element of order 2, whence we must have $\phi_p =1$ and $E(\Q_p)[2] = E[2]$. The assertions of the theorem are now clear from \eqref{cr} and \eqref{cr'}.
\end{proof}

\medskip

Finally, we now compare some of our estimates with those predicted by the conjecture of Birch and Swinnerton-Dyer. The next proposition is an immediate consequence of  Theorems \ref{18} and  \ref{26}.

\begin{prop}\label{bsd} Assume that $E$ is defined over $\Q$ and has good reduction at 2, and that $K \neq \Q(\sqrt{-3})$. Assume further that $L(E, 1) \neq 0$, and that $ord_2(L^{(alg)}(E, 1)) < 0$. Then, for all positive integers $M$, which are admissible for $E$, and have the property that all of their prime factors are $\equiv 1 \, mod \, 4$, we have
\begin{equation}\label{28}
ord_2(\frac{L^{(alg)}(E^{(M)}, 1)}{L^{(alg)}(E, 1)}) \geq ord_2(\displaystyle\prod_{p|M} {c_p(E^{(M)})}).
\end{equation}
\end{prop}

As we shall now explain, the lower bound given by \eqref{28} is exactly what the conjecture of Birch and Swinnerton-Dyer would predict for elliptic curves satisfying the hypotheses of this proposition. We first establish a preliminary result.

\begin{prop}\label{29} Let $E$ be an elliptic curve defined over $\Q$, with complex multiplication by the ring of integers of $K$. Assume that $E$ has good reduction at 2, and that $K \neq \Q(\sqrt{-3})$. Let  $M$ denote an odd positive square free integer with $(M, N(E)) =1$, and put $\epsilon$ equal to $+1$ or $-1$, according as $M$ is congruent to 1 or 3 $mod \, 4$. Then the 2-primary subgroups of
$E(\Q)$ and $E^{(\epsilon M)}(\Q)$ have the same order, and this order is equal to 2 or 1, according as the prime 2 splits or is inert in $K$.
\end{prop}
\begin{proof} Let $A$ denote the elliptic curve $E$ or $E^{(\epsilon M)}(\Q)$, so that $A$ also has good reduction at $2$. In order to show that the 2-primary subgroup of $A(\Q)$ is annihilated by 2,
it suffices to prove that the 2-primary subgroup of $A(K)$ is annihilated by 2. Now, as $E$ has good reduction at 2, the prime 2 does not ramify in $K$, and thus it either splits or is inert in $K$. Let $v$ denote any prime of $K$ above 2. Since $A$ has good reduction at $v$, the formal group of $A$ at $v$
is a Lubin-Tate formal group with parameter $\pi = \psi_A(v)$. Let $n$ be any integer $\geq 1$. As the group $A[{\pi^n}]$ of $\pi^n$-division points on $A$ lies on the formal group of $A$ at $v$, it follows from Lubin-Tate theory the extension $K(A[{\pi^n}])/K$ has Galois group isomorphic to $(\cO/\pi^n\cO)^*$, which is non-trivial for all $n\geq 1$ if $2$ is inert in $K$, and which is non-trivial for all $n \geq 2$ if 2 splits in $K$. In particular, the 2-primary subgroup of $A(K)$ must be trivial if 2 is inert in $K$, and it must be killed by 2 when 2 splits in $K$. But 2 splits in $K$ happens precisely when $K = \Q(\sqrt{-7})$, and then $A$ must be a quadratic twist of the curve given by $\eqref{22}$.
Now the curve $\eqref{22}$ has a unique rational point of order 2 given by $(2, -1)$. It follows that
$A(\Q)$ must also have a unique point of order 2, because $A$ is a quadratic twist of $\eqref{22}$.
This completes the proof.
\end{proof}

Now assume that $E$ satisfies the hypotheses of Proposition \ref{bsd}. Since $L(E, 1)\neq 0$, we know that both $E(\Q)$ and the Tate-Shafarevich group
of $E/\Q$ are finite, and we write $w(E)$ and $t(E)$ for their respective orders. Then the conjecture of Birch and Swinnerton-Dyer predicts that
\begin{equation}\label{30}
ord_2(L^{(alg)}(E, 1)) = ord_2(c_\infty(E)\displaystyle\prod_{p|N(E)} {c_p(E)}) + ord_2(t(E)) - 2ord_2(w(E)).
\end{equation}
where $c_\infty(E)$ denotes the number of connected components of $E(\mathbb{R})$. If we recall Proposition \ref{29}, and the fact that the Cassels-Tate theorem implies that $t(E)$ is the square of an integer, we see that the combination of our hypothesis that $ord_2(L^{(alg)}(E, 1)) < 0$ and the conjectural formula \eqref{30} imply that necessarily
\begin{equation}\label{31}
ord_2(t(E)) = 0.
\end{equation}
Suppose now that $L(E^{(M)}, 1)\neq 0$. Again, we then know that both $E^{(M)}(\Q)$ and the Tate-Shafarevich group
of $E^{(M)}/\Q$ are finite, and we write $w(E^{(M)})$ and $t(E^{(M)})$ for their respective orders. Then, in this case, the conjecture of Birch and Swinnerton-Dyer predicts that
\begin{equation}\label{32}
ord_2(L^{(alg)}(E^{(M)}, 1)) = ord_2(c_\infty(E^{(M)})\displaystyle\prod_{p|N(E)M} {c_p(E^{(M)})}) + ord_2(t(E^{(M)})) - 2ord_2(w(E^{(M)})).
\end{equation}
where $c_\infty(E^{(M)})$ denotes the number of connected components of $E^{(M)}(\mathbb{R})$.
Obviously, $c_\infty(E)= c_\infty(E^{(M)})$ since $\Q(\sqrt{M})$ is a real quadratic field.  Moreover, Proposition \ref{29} shows that $ord_2(w(E))=ord_2(w(E^{(M)}))$, and Lemma \ref{eq} tells us that, for primes $p$ dividing $N(E)$, we have $ ord_2(c_p(E)) =  ord_2(c_p(E^{(M)})$. Hence, recalling \eqref{31}, we conclude that the conjecture of Birch and Swinnerton-Dyer predicts that
\begin{equation}\label{33}
ord_2(\frac{L^{(alg)}(E^{(M)}, 1)}{L^{(alg)}(E, 1)}) = ord_2(\displaystyle\prod_{p|M} {c_p(E^{(M)})}) +
ord_2(t(E^{(M)})).
\end{equation}
This shows that, under the above hypotheses, the lower bound given by \eqref{28} is precisely what the conjecture of the Birch and Swinnerton-Dyer would predict if we ignore the unknown term  $ord_2(t(E^{(M)}))$ giving the order of the 2-primary subgroup of the Tate-Shafarevich group of the curve $E^{(M)}$.

\section{Tables}

In this section, we include some short tables of numerical examples of our results for two elliptic curves $E$ defined over $\Q$. We use the same notation as earlier. For the curve of conductor 49 in Table I, the root number of the curve is $+1$, and for the curve of conductor 121 in Table II the root number is $-1$. As always, $M$ will denote a square free positive integer which is admissible for the elliptic curve $E$, and $r(M)$ will denote the number of prime divisors of $M$ in the field of complex multiplication $K$.

\begin{center}
\tablefirsthead{%

 \hline

\multicolumn{6}{|c|}{Table I:  Case $X_0(49): y^2+xy=x^3-x^2-2x-1$.}\\
\multicolumn{6}{|c|}{The  Tamagawa factor of twists $E^{(M)}$ at $7$ is always $2$.}\\
\multicolumn{6}{|c|}{$L^{(alg)}(E,1)=1/2$.}\\
\hline  $M$&$L(E^{(M)},1)$&$L^{(alg)}(E^{(M)},1)$&$ord_2L^{(alg)}(E^{(M)},1)$&$r(M)$&Tamagawa Factors \\
\hline }

 \tablehead{\hline \multicolumn{6}{|c|}{Table I:  Case $X_0(49): y^2+xy=x^3-x^2-2x-1$.}\\

\hline  $M$&$L(E^{(M)},1)$&$L^{(alg)}(E^{(M)},1)$&$ord_2L^{(alg)}(E^{(M)},1)$&$r(M)$&Tamagawa Factors \\
\hline }

\tabletail{\hline}

 \tablelasttail{\hline}
\begin{supertabular}{|l@{\hspace{0pt}}|l@{\hspace{0pt}}|l@{\hspace{0pt}}|l@{\hspace{0pt}}|l@{\hspace{0pt}}|l@{\hspace{0pt}}|}
$29$&$0.7180139420$&$2$&$1$&$2$&$c_{29}=4$,\\

$37$&$0.6356689731$&$2$&$1$&$2$&$c_{37}=4$,\\

$109$&$0.3703553538$&$2$&$1$&$2$&$c_{109}=4$,\\

$113$&$1.454965333$&$8$&$3$&$2$&$c_{113}=4$,\\

$137$&$0.3303479321$&$2$&$1$&$2$&$c_{137}=4$,\\

$145$&$0.6422111932$&$4$&$2$&$3$&$c_{5}=2$, $c_{29}=4$,\\

$185$&$2.274238456$&$16$&$4$&$3$&$c_{5}=2$, $c_{37}=4$,\\

$233$&$2.279798298$&$18$&$1$&$2$&$c_{233}=4$,\\

$265$&$4.275446184$&$36$&$2$&$3$&$c_{5}=2$, $c_{53}=4$,\\

$277$&$0.9292915388$&$8$&$3$&$2$&$c_{277}=4$,\\

$281$&$0.2306634143$&$2$&$1$&$2$&$c_{281}=4$,\\

$285$&$1.832312031$&$16$&$4$&$3$&$c_{3}=2$, $c_{5}=2$, $c_{19}=2$,\\

$317$&$0.8686848279$&$8$&$3$&$2$&$c_{317}=4$,\\

$337$&$0.2106283985$&$2$&$1$&$2$&$c_{337}=4$,\\

$377$&$0.3982824745$&$4$&$2$&$3$&$c_{13}=2$, $c_{29}=4$,\\

$389$&$1.764410302$&$18$&$1$&$2$&$c_{389}=4$,\\

$401$&$1.737809629$&$18$&$1$&$2$&$c_{401}=4$,\\

$421$&$0.7537907774$&$8$&$3$&$2$&$c_{421}=4$,\\

$449$&$2.919635854$&$32$&$5$&$2$&$c_{449}=4$,\\

$457$&$0.7234920569$&$8$&$3$&$2$&$c_{457}=4$,\\

$481$&$1.410422816$&$16$&$4$&$3$&$c_{13}=2$, $c_{37}=4$,\\

$545$&$0.3312558988$&$4$&$2$&$3$&$c_{5}=2$, $c_{109}=4$,\\

$557$&$0.6553363680$&$8$&$3$&$2$&$c_{557}=4$,\\

$565$&$0.3253401390$&$4$&$2$&$3$&$c_{5}=2$, $c_{113}=4$,\\

$569$&$0.1620972858$&$2$&$1$&$2$&$c_{569}=4$,\\

$613$&$0.1561714487$&$2$&$1$&$2$&$c_{613}=4$,\\

$617$&$0.1556643972$&$2$&$1$&$2$&$c_{617}=4$,\\

$629$&$1.233378974$&$16$&$4$&$3$&$c_{17}=2$, $c_{37}=4$,\\

$641$&$0.1527224426$&$2$&$1$&$2$&$c_{641}=4$,\\

$653$&$0.1513126668$&$2$&$1$&$2$&$c_{653}=4$,\\

$673$&$1.341426413$&$18$&$1$&$2$&$c_{673}=4$,\\

$701$&$0.1460403507$&$2$&$1$&$2$&$c_{701}=4$,\\

$705$&$1.165003700$&$16$&$4$&$3$&$c_{3}=2$, $c_{5}=2$, $c_{47}=2$,\\

$709$&$0.1452140903$&$2$&$1$&$2$&$c_{709}=4$,\\

$757$&$0.1405348183$&$2$&$1$&$2$&$c_{757}=4$,\\

$809$&$0.5437729586$&$8$&$3$&$2$&$c_{809}=4$,\\

$821$&$0.5397843500$&$8$&$3$&$2$&$c_{821}=4$,\\

$877$&$1.175099358$&$18$&$1$&$2$&$c_{877}=4$,\\

$901$&$1.030527220$&$16$&$4$&$3$&$c_{17}=2$, $c_{53}=4$,\\

$953$&$0.5010088727$&$8$&$3$&$2$&$c_{953}=4$,\\

$965$&$0.2489420234$&$4$&$2$&$3$&$c_{5}=2$, $c_{193}=4$,\\

$969$&$0.9937107192$&$16$&$4$&$3$&$c_{3}=2$, $c_{17}=2$, $c_{19}=2$,\\

$977$&$1.113338183$&$18$&$1$&$2$&$c_{977}=4$,\\

$985$&$2.217615590$&$36$&$2$&$3$&$c_{5}=2$, $c_{197}=4$,\\

\end{supertabular}
\end{center}

\begin{center}
\tablefirsthead{%

 \hline

\multicolumn{6}{|c|}{Table II:  Case $E: y^2+y=x^3-x^2-7x+10$ of conductor 121.}\\
\multicolumn{6}{|c|}{The  Tamagawa factor of twists $E^{(-M)}$ at $11$ is always $2$.}\\

\hline  $M$&$|L(E^{(-M)},1)|$&$|L^{(alg)}(E^{(-M)},1)|$&$ord_2|L^{(alg)}(E^{(-M)},1)|$&$r(M)$&Tamagawa factors \\
\hline }

 \tablehead{\hline \multicolumn{6}{|c|}{Table II:  Case $X_0(121): y^2+y=x^3-x^2-7x+10$.}\\

\hline  $M$&$|L(E^{(-M)},1)|$&$|L^{(alg)}(E^{(-M)},1)|$&$ord_2|L^{(alg)}(E^{(-M)},1)|$&$r(M)$&Tamagawa factors  \\
\hline }

\tabletail{\hline}

 \tablelasttail{\hline}
\begin{supertabular}{|l@{\hspace{0pt}}|l@{\hspace{0pt}}|l@{\hspace{0pt}}|l@{\hspace{0pt}}|l@{\hspace{0pt}}|l@{\hspace{0pt}}|}
$7$&$1.094573405$&$4$&$2$&$1$&$c_{7}=2$,\\

$43$&$0.4416311353$&$4$&$2$&$1$&$c_{43}=2$,\\

$79$&$0.3258219706$&$4$&$2$&$1$&$c_{79}=2$,\\

$83$&$0.3178738964$&$4$&$2$&$1$&$c_{83}=2$,\\

$107$&$0.2799638923$&$4$&$2$&$1$&$c_{107}=2$,\\

$119$&$0.5309460896$&$8$&$3$&$2$&$c_{7}=2$, $c_{17}=2$,\\

$127$&$1.027902784$&$16$&$4$&$1$&$c_{127}=2$,\\

$131$&$0.2530219881$&$4$&$2$&$1$&$c_{131}=2$,\\

$139$&$0.2456328864$&$4$&$2$&$1$&$c_{139}=2$,\\

$151$&$0.2356706165$&$4$&$2$&$1$&$c_{151}=2$,\\

$203$&$0.4065143570$&$8$&$3$&$2$&$c_{7}=2$, $c_{29}=2$,\\

$211$&$0.7974669169$&$16$&$4$&$1$&$c_{211}=2$,\\

$227$&$0.1922122148$&$4$&$2$&$1$&$c_{227}=2$,\\

$239$&$0.1873246635$&$4$&$2$&$1$&$c_{239}=2$,\\

$247$&$0.3685321923$&$8$&$3$&$2$&$c_{13}=2$, $c_{19}=2$,\\

$263$&$0.1785730998$&$4$&$2$&$1$&$c_{263}=2$,\\

$271$&$0.7036703591$&$16$&$4$&$1$&$c_{271}=2$,\\

$287$&$0.3418872925$&$8$&$3$&$2$&$c_{7}=2$, $c_{41}=2$,\\

$307$&$2.644506912$&$64$&$6$&$1$&$c_{307}=2$,\\

$323$&$0.3222720533$&$8$&$3$&$2$&$c_{17}=2$, $c_{19}=2$,\\

$347$&$0.6218550501$&$16$&$4$&$1$&$c_{347}=2$,\\

$371$&$0.6014048805$&$16$&$4$&$3$&$c_{7}=2$, $c_{53}=4$,\\

$427$&$0.2802915271$&$8$&$3$&$2$&$c_{7}=2$, $c_{61}=2$,\\

$431$&$0.1394939193$&$4$&$2$&$1$&$c_{431}=2$,\\

$439$&$0.5528682408$&$16$&$4$&$1$&$c_{439}=2$,\\

$491$&$0.5227730093$&$16$&$4$&$1$&$c_{491}=2$,\\

$503$&$0.1291248765$&$4$&$2$&$1$&$c_{503}=2$,\\

$511$&$0.2562202539$&$8$&$3$&$2$&$c_{7}=2$, $c_{73}=2$,\\

$547$&$1.981163103$&$64$&$6$&$1$&$c_{547}=2$,\\

$551$&$0.2467448562$&$8$&$3$&$2$&$c_{19}=2$, $c_{29}=2$,\\

$559$&$0.2449728774$&$8$&$3$&$2$&$c_{13}=2$, $c_{43}=2$,\\

$563$&$0.4882021701$&$16$&$4$&$1$&$c_{563}=2$,\\

$607$&$1.057893809$&$36$&$2$&$1$&$c_{607}=2$,\\

$659$&$0.1128109364$&$4$&$2$&$1$&$c_{659}=2$,\\

$707$&$0.8713129959$&$32$&$5$&$2$&$c_{7}=2$, $c_{101}=2$,\\

$731$&$0.2142225669$&$8$&$3$&$2$&$c_{17}=2$, $c_{43}=2$,\\

$739$&$0.1065299425$&$4$&$2$&$1$&$c_{739}=2$,\\

$743$&$0.4249711969$&$16$&$4$&$1$&$c_{743}=2$,\\

$763$&$0.8387289424$&$32$&$5$&$2$&$c_{7}=2$, $c_{109}=2$,\\

$787$&$1.651682351$&$64$&$6$&$1$&$c_{787}=2$,\\

$811$&$0.9152210367$&$36$&$2$&$1$&$c_{811}=2$,\\

$887$&$0.09723712323$&$4$&$2$&$1$&$c_{887}=2$,\\

$919$&$0.09552920333$&$4$&$2$&$1$&$c_{919}=2$,\\

\end{supertabular}
\end{center}

\footnotesize{ J.C:  Emmanuel College, Cambridge, England, and Department of
Mathematics, POSTECH, Pohang, South Korea, email:
\texttt{jhc13@dpmms.cam.ac.uk}}

\footnotesize{M.K: Merton College, Oxford, England, email:
\texttt{Minhyong.Kim@maths.ox.ac.uk}}

\footnotesize{Z.L:
School of Mathematical Sciences, Capital Normal University, and
Beijing International Center for Mathematical Research, Peking University,
Beijing, People's Republic of China,
email: \texttt{liangzhb@gmail.com}}

\footnotesize{C.Z.:
Department of Mathematics, Peking University,  Beijing,
People's Republic of China,
email: \texttt{zhao@math.pku.edu.cn}}

\end{document}